\documentclass[12pt]{amsart}
\usepackage{amsthm}
\usepackage{amstext}
\usepackage{amssymb}
\usepackage{amsrefs}
\usepackage{paralist}

\numberwithin{equation}{section}
\numberwithin{figure}{section}

\newtheorem{theorem}{Theorem}[section]
\newtheorem{lemma}[theorem]{Lemma}
\theoremstyle{definition}
\newtheorem{definition}[theorem]{Definition}
\newtheorem{example}[theorem]{Example}

\theoremstyle{remark}
\newtheorem{remark}[theorem]{Remark}
\numberwithin{equation}{section}
\theoremstyle{plain}

\newtheorem{conjecture}[theorem]{Conjecture}
\newtheorem{corollary}[theorem]{Corollary}

\newtheorem{question}[theorem]{Question}
\newtheorem{proposition}[theorem]{Proposition}

\newcommand{\R}{\mathbb{R}}

\newcommand{\ol}{\overline}
\renewcommand{\div}{\text{div}}
\newcommand{\divm}{\text{\emph{div}}}

\DeclareMathOperator{\tr}{tr}

\renewcommand{\R}{\mathbb{R}}

\newcommand{\II}{\operatorname{{I\hspace{-0.2mm}I}}}
\newcommand{\trSigmak}{\tr_\Sigma\!k}

\begin{document}

\begin{abstract}
We identify a condition on spacelike 2-surfaces in a spacetime that is relevant to understanding the concept of mass in general relativity.  We prove a formula for the variation of the spacetime Hawking mass under a uniformly area expanding flow and show that it is nonnegative for these so-called ``time flat surfaces.''  Such flows generalize inverse mean curvature flow, which was used by Huisken and Ilmanen to prove the Riemannian Penrose inequality for one black hole.
A flow of time flat surfaces may have connections to the problem in general relativity of bounding the mass of a spacetime from below by the quasi-local mass of a spacelike 2-surface contained therein.
\end{abstract}

\title[Time flat surfaces and the monotonicity of the  Hawking mass]{Time flat surfaces and  the monotonicity of the spacetime Hawking mass}
\author{Hubert L. Bray}
\address{Dept. of Mathematics,
Duke University,
Durham, NC 27708}
\email{bray@math.duke.edu}

\author{Jeffrey L. Jauregui}
\address{Dept. of Mathematics,
Union College,
Schenectady, NY 12308}
\email{jaureguj@union.edu}

\date{\today}
\maketitle

\section{Introduction} 
In general relativity a significant and influential problem is to understand  how the quasi-local mass of a spacelike surface contained in a (3+1)-dimensional spacetime $(N, \langle\cdot,\cdot\rangle)$ provides a lower bound for the total mass of $N$.  As a special case, the conjectured Penrose inequality predicts that the ADM mass is bounded from below in terms of the area of outermost apparent horizons of black holes, subject to natural hypotheses.

Very important progress on the above problem was made by Huisken and Ilmanen in 2001, in the case $N$ admits a totally geodesic (``time-symmetric''), asymptotically flat spacelike hypersurface $(M^3,g)$ \cite{imcf}.  Their motivation was prior work of  Geroch \cite{geroch},
Jang \cites{jang1, jang2, jang3}, and Jang--Wald \cites{jang_wald} in the 1970s, who discovered the connection between inverse mean curvature flow (IMCF) and the Hawking mass.  A smooth family of closed surfaces $\{\Sigma_\lambda\}$ in $(M,g)$ is said to satisfy IMCF if their velocity is outward-normal with speed equal to the reciprocal of their mean curvature, $H$.  The Hawking mass of a surface $\Sigma$ in $M$ is defined to be
$$m_H(\Sigma) = \sqrt{\frac{|\Sigma|}{16\pi}} \left( 1 - \frac{1}{16\pi} \int_\Sigma H^2 dA\right),$$
where $|\Sigma|$ is the area of $\Sigma$.
The key observation of Geroch, Jang, and Wald is that if $\Sigma_\lambda$ are connected and satisfy IMCF,
then $m_H(\Sigma_\lambda)$ is nondecreasing, provided $(M,g)$ has nonnegative scalar curvature (which itself follows from the dominant energy condition on $N$).  Moreover, they realized that such a flow of surfaces could be useful for bounding the total mass of $M$ in terms of the Hawking mass of $\Sigma$.  However, it has long been known that smooth solutions to IMCF with given initial condition may not exist.  The contribution of Huisken and Ilmanen was to develop a theory of weak solutions to IMCF that maintained all the key properties of the smooth case.  As a corollary, the so-called Riemannian Penrose inequality (for a single black hole) followed.
(Bray proved the multiple black hole case using different techniques \cite{bray_RPI}.)

However, a spacetime generally has no such totally geodesic hypersurface, so it is highly desirable to achieve similar results without that hypothesis.  In the literature, there appear to be two broad  approaches to doing so.  First, one could take an asymptotically flat hypersurface $(M,g)$,  treat its extrinsic curvature $k$ as auxiliary data, and attempt to define a flow purely within $M$ (see, for instance, \cites{bray_khuri, MMS, moore}). 
Second, one could construct a codimension-two flow of surfaces within $N$ (see, for instance, \cites{frauendiener, hayward, imcf_spacetime}).  In the former approach, no appropriate monotone quantity is known; in the latter, finding an existence theory has proven elusive. For instance, surfaces evolving with velocity \cite{frauendiener}
\begin{equation}
\label{eqn_imcvf}
\vec \xi =\frac{-\vec H}{\langle \vec H, \vec H\rangle},
\end{equation}
where $\vec H$ is the mean curvature vector of $\Sigma$ in $N$,
satisfy a forwards-backwards parabolic PDE system (and thus solutions for most  initial data do not exist).  However, any solution does have nondecreasing  (spacetime) Hawking mass \cite{hawking}
\begin{equation}
\label{eqn_hawking}
m_H(\Sigma) = \sqrt{\frac{|\Sigma|}{16\pi}} \left(1 - \frac{1}{16\pi} \int_{\Sigma} \langle \vec H, \vec H \rangle dA\right)
\end{equation}
when $N$ satisfies the dominant energy condition and the surfaces are connected.

In this paper, we take the second approach, studying the variation of the Hawking mass under \emph{uniformly area expanding flows} in $N$, and define the notion of a \emph{time flat} surface.
A flow is uniformly area expanding if its velocity $\vec \xi$ can be written in the form
\begin{equation}
\label{eqn_xi}
\vec \xi = \frac{-\vec H}{\langle \vec H, \vec H\rangle} + \beta\frac{-\vec H^\perp}{\langle \vec H, \vec H\rangle},
\end{equation}
for some function $\beta$ on $\Sigma$, where $\vec H ^\perp$ is orthogonal to $\Sigma$ and $\vec H$ and has the same length as $\vec H$ (cf. Definition \ref{def_uae}).  As such, uniformly area expanding flows generalize inverse mean curvature vector flow (\ref{eqn_imcvf}) and were studied in \cites{MMS, imcf_spacetime}.
Given the success of inverse mean curvature flow in the time-symmetric case, it is natural to study the variation of the Hawking mass under a uniformly area expanding flow.

Our main result is the following new formula for the variation of the Hawking mass, which builds on previous work of Malec, Mars, and Simon \cite{MMS} (see also \cite{imcf_spacetime}).
\begin{theorem}
\label{thm_tf_monotonicity}
Let $\Sigma$ be a closed spacelike surface embedded in a spacetime $N$, having spacelike mean curvature vector $\vec H$. Let $\Sigma_\lambda$ be a smooth family of surfaces with $\Sigma_0 = \Sigma$ that is uniformly area expanding at $\lambda=0$, with outward-spacelike velocity $\vec \xi$.  Then:
\begin{align*}
\left.\frac{d m_H(\Sigma_\lambda)}{d\lambda}\right|_{\lambda=0} &= \sqrt{\frac{|\Sigma|}{(16\pi)^3}}  \Bigg\{4\pi(2-\chi(\Sigma))+2\int_\Sigma  G(-\vec H^\perp, \vec \xi^\perp)dA\\
&\;+\int_\Sigma \left(|\mathring \II_{\nu_H} |^2 + 2\beta \langle  \mathring \II_{\nu_H}, \mathring \II_{\nu_H^\perp}\rangle + |\mathring \II_{\nu_H^\perp}|^2 \right)dA\\
&\;+2\int_\Sigma \left(\left| \frac{\nabla^\Sigma |\vec H|}{|\vec H|}\right|^2 + 2\beta\alpha_H\left(\frac{\nabla^\Sigma |\vec H|}{|\vec H|}\right) +|\alpha_H|^2 + \beta\divm_\Sigma(\alpha_H)\right)dA\Bigg\}
 \end{align*}
\end{theorem}

All notation is explained in sections \ref{sec_monotonicity} and \ref{sec_time_flat}. Note that $\beta$ determines the flow velocity according to (\ref{eqn_xi}) and that
$|\beta| < 1$ since $\vec \xi$ and $\vec H$ are spacelike.

\begin{remark}
In the first  version of this paper, posted to the arXiv, the above formula was stated incorrectly (and in a different form) due to a sign error which we are indebted to Marc Mars for pointing out to us.  
\end{remark}

To motivate the definition of time flat, we will show that if $\div(\alpha_H)$ vanishes, then above formula for the variation of the Hawking mass is nonnegative, if $|\beta| < 1$ and the dominant energy condition is satisfied.

\begin{definition}
\label{def_TF}
A two-dimensional, embedded spacelike surface $\Sigma$ with spacelike mean curvature vector $\vec H$ in a (3+1)-dimensional spacetime is \emph{time flat} if 
$\div_\Sigma (\alpha_H) = 0$, where $\alpha_H$ is the connection 1-form of $T^\perp\Sigma$ associated to the unit normal $\nu_H = -\frac{\vec H}{|\vec H|}$ (cf. equation (\ref{eqn_conn_form})), and $\div_\Sigma$ is the divergence on $\Sigma$.
\end{definition}
In section \ref{sec_time_flat} we explain why the time flat condition is geometrically natural and how it relates to the Wang--Yau quasi-local mass \cites{wang_yau1, wang_yau2}.  The definition extends  to spacelike $(n-1)$-dimensional submanifolds $\Sigma$ of an $(n+1)$-dimensional Lorentzian spacetime, for $n \geq 2$.  For $n=2$, time flatness is equivalent to the curve $\Sigma$ having constant torsion.

While the variation of the Hawking mass can have any sign in general, it is always nonnegative on time-flat surfaces:

\begin{corollary}
\label{cor_tf}
Let $\Sigma$ be a closed, connected time flat surface embedded in a spacetime $N$ that  obeys the dominant energy condition.  Then for any outward-spacelike, uniformly area expanding flow
$\{\Sigma_\lambda\}$ with $\Sigma_0 = \Sigma$,
$$\left.\frac{d}{d\lambda} m_H(\Sigma_\lambda)\right|_{\lambda=0} \geq 0.$$
\end{corollary}

For establishing intuition, we present the following picture.Ê The Hawking mass of a round sphere in the $t=0$ slice of the Minkowski spacetime (representing vacuum) is zero.Ê Introducing spatial oscillations to this sphere makes the Hawking mass negative (that is, too small), whereas introducing timelike oscillations to this sphere can make the Hawking mass positive (that is, too big; see Example \ref{ex_sphere}).Ê Now suppose we want to find a flow of surfaces that makes the Hawking mass nondecreasing and convergent to the total (ADM) mass at infinity.Ê Since the ADM mass of Minkowski space is zero, this is not possible if we start with a surface with positive Hawking mass.

Hence, the only way to find a flow of surfaces where the Hawking mass is nondecreasing and convergent to the ADM mass at infinity is to not allow the flow to begin with surfaces with Hawking mass that is too large in the first place.Ê Direct computation suggests that ``timelike oscillations'' in a surface can make the Hawking mass of the surface too large.Ê We suggest that time flat surfaces, defined above, should be thought of as surfaces without timelike oscillations.Ê 
In section \ref{sec_discussion} we describe a uniformly area expanding flow that preserves the time flat condition with the hope of keeping timelike oscillations in check, and possibly allowing the Hawking masses to converge to the ADM mass.

\vspace{2mm}
\paragraph{\emph{Outline}}
Section \ref{sec_setup} introduces conventions and notation.  In section \ref{sec_connection}, we study a natural variational problem on the normal bundle of $\Sigma$ that defines a notation of ``straight out'' unit spacelike direction from $\Sigma$.  This will provide  motivation for the  divergence-free condition on the connection 1-form appearing in the definition of time flat. 
 In section \ref{sec_monotonicity} we prove in full detail a variation formula for the Hawking mass that first appeared in \cite{MMS}. 
Section \ref{sec_time_flat} is the crux of the paper, in which the time flat condition is explored and Theorem \ref{thm_tf_monotonicity} is proved.  We conclude with section \ref{sec_discussion}, a discussion on a flow of time flat surfaces and its possible applications to problems pertaining to bounding total mass in terms of  quasi-local mass.

\vspace{2mm}
\paragraph{\emph{Acknowledgements}} The authors are extremely grateful to Marc Mars for identifying a sign error in the first version of this preprint which led us to an improved version of Theorem 1.1. The first named author was supported in part by NSF grant \#DMS-1007063.  This material is based upon work supported by the NSF under grant \#DMS-0932078 000, while the second named author was in residence at the Mathematical Sciences Research Institute in Berkeley, California during the fall of 2013.

\section{Setup}
\label{sec_setup}
For the remainder of the paper, we shall assume  $(N,\langle \;,\;\rangle)$ is a connected, smooth, time-oriented, four-dimensional spacetime of signature $(-,+,+,+)$ that obeys Einstein's equation
\begin{equation}
\label{eqn_einstein}
G = 8\pi T,
\end{equation}
where $G$ and $T$ are, respectively, the Einstein curvature tensor and stress-energy tensor of $N$. All manifolds, functions, tensors, are assumed to be smooth unless stated otherwise.

\vspace{2mm}
\paragraph{\emph{Geometry of spacelike hypersurfaces}}
Suppose $M$ is a spacelike hypersurface in $N$, with future-pointing unit normal $\vec n$, induced Riemannian metric $g$, and second fundamental form $k$ in the direction of $\vec n$ (see equation (\ref{eqn_k}) for the sign convention).  The energy and momentum densities relative to $M$ are defined respectively by
\begin{align*}
\mu &= T(\vec n, \vec n),\\
J(X) &= T(X, \vec n),
\end{align*}
for any tangent vector $X$ to $M$.  The Gauss and Codazzi equations, together with Einstein's equation (\ref{eqn_einstein}) imply the \emph{constraint equations} on $M$:
\begin{align}
16\pi \mu &= R + (\tr_M\!k)^2 - |k|^2 \label{constraint1}\\
8\pi J &=  \div_M(k-(\tr_M\! k)g) \label{constraint2},
\end{align}
where $R$ is the scalar curvature of $g$, $\tr_M\! k$ is the trace of $k$ with respect to $g$, $|k|^2$ is the norm-squared of $k$ with respect to $g$, and $\div_M$ is the divergence operator with respect to $g$, acting on symmetric $(0,2)$-tensors.

Recall that $N$ satisfies the \emph{dominant energy condition} if $T(u,v) \geq 0$ for all future-pointing timelike vectors $u,v$ in $N$ based at the same point.  This implies that
\begin{equation}
\label{DEC}
\mu \geq J(X)
\end{equation}
for any tangent vector $X$ to $M$ of length at most 1.

\vspace{2mm}
\paragraph{\emph{The normal bundle of a surface}}
Let $\Sigma$ be a closed 2-manifold, 
and let $\Phi: \Sigma \to N$ be an embedding such that $\Sigma_0 := \Phi(\Sigma)$ is spacelike.  We shall always assume the notion of inward- and outward-pointing spacelike vectors to $\Sigma$ is well-defined; this is the case, for instance, if $\Sigma_0$ is the boundary of a compact three-dimensional region
contained in a noncompact, complete spacelike hypersurface in $N$.  We shall identify $\Sigma_0$ with $\Sigma$.

Let $E$ be the rank-four vector bundle over $\Sigma$ given by restricting $TN$.  
We have an orthogonal decomposition
$$E = T\Sigma \oplus T^\perp \Sigma,$$
where $T^\perp \Sigma$ is the normal bundle of $\Sigma$, with induced metric of signature $(-,+)$.   Note that the set of non-null vectors in each fiber of $T^\perp \Sigma$ is partitioned into four quadrants: the future-timelike, past-timelike, outward-spacelike, and inward-spacelike vectors.

There exists a natural, involutive, linear isomorphism $u \mapsto u^\perp$
defined on the fibers of $T^\perp \Sigma$ as follows (and analogous to a $90^\circ$ rotation in a Euclidean plane).  For $u, v\in T_p^\perp \Sigma$ comprising
an orthonormal basis, with $u$ future-timelike and $v$ outward-spacelike, define $u^\perp=v$ and $v^\perp=u$ and extend linearly. This definition is basis-independent. Moreover, for any $w \in T_p^\perp \Sigma$,
$w^\perp $ is orthogonal to $w$ and their norms-squared have opposite signs:
$$-\langle w^\perp, w^\perp \rangle = \langle w, w \rangle.$$

\vspace{2mm}
\paragraph{\emph{Sign conventions}} Given a semi-Riemannian submanifold $P$ of a semi-Riemannian manifold $(L,\langle \cdot, \cdot \rangle)$ (with either possibly Riemannian), the second fundamental form
$\vec\II$ of $P$ is:
$$\vec\II(X,Y) = \nabla^L_X Y - \nabla^P_X Y,$$
where $\nabla^L$ and $\nabla^P$ are the respective Levi-Civita connections on $L$ and $P$. The mean curvature vector $\vec H$ is the trace of $\vec \II$ with respect to the induced metric on $P$.  Given a normal vector field $\vec n$ to $P$, 
we define the second fundamental form of $P$ in the direction $\vec n$ as
\begin{equation}
\label{eqn_k}
k = -\langle \vec \II, \vec n\rangle
\end{equation}
and the mean curvature in the direction $\vec n$ as
\begin{equation}
\label{eqn_H}
H = -\langle \vec H, \vec n \rangle.
\end{equation}
If $P$ has codimension one and we take $\epsilon=\langle\vec n, \vec n\rangle=\pm 1$,
these definitions are equivalent to:
$$\vec \II = -\epsilon k \vec n \qquad \text{and} \qquad \vec H = -\epsilon H \vec n$$

Thus, for instance, a round sphere in $\R^3$ has inward-pointing mean curvature vector, and positive mean curvature in the outward direction.

\section{The normal connection}
\label{sec_connection}
In this section, we consider a variational problem on the normal bundle of $\Sigma$ whose purpose is to select a natural choice of outward-spacelike unit vector field $\nu$.  The point is to find a $\nu$ that varies as little as possible over $\Sigma$; such $\nu$ could be called a ``straight out'' direction from $\Sigma$.  Indeed, $\nu$ will be parallel whenever a parallel section exists, and in all cases is unique up to hyperbolic rotations by constant angle.

Let $\nabla^\perp$ be the induced connection on $T^\perp \Sigma$,  given by projecting $\nabla^N$ orthogonally onto $T^\perp \Sigma$, fiberwise, where $\nabla^N$ is the Levi-Civita connection of $N$.

\vspace{2mm}
\paragraph{\emph{The connection 1-form}}
Fix any section $\nu$ of $T^\perp \Sigma$ that is outward-spacelike and of unit length.  Then for any tangent vector field $X$  to $\Sigma$, $0=D_X \langle \nu, \nu \rangle =2 \langle \nabla^\perp_X \nu, \nu\rangle$, so $\nabla^\perp_X \nu$ is a multiple of  $\nu^\perp$.
Thus, there exists a unique 1-form $\alpha_\nu$ on $\Sigma$, depending on $\nu$,
so that
\begin{equation}
\label{eqn_conn_form}
\langle\nabla_{X}^\perp \nu, \nu^\perp\rangle = \alpha_\nu(X).
\end{equation}
and similarly,
$$\langle\nabla_{X}^\perp \nu^\perp, \nu\rangle = -\alpha_\nu(X).$$
Observe that $\nu$ and $\alpha_\nu$  completely determine $\nabla^\perp$ and that $\alpha_\nu$ vanishes if and only if $\nu$ is parallel.
We shall write $\alpha$ in place of $\alpha_\nu$ when there is no chance of confusion.  

\vspace{2mm}
\paragraph{\emph{Change-of-basis}}
If $\ol \nu$ is some other choice of outward-spacelike unit normal to $\Sigma$, say with connection 1-form $\ol \alpha$, we can write:
\begin{align}
\ol \nu &= \cosh(\theta) \nu + \sinh(\theta) \nu^\perp \label{eqn_nu_nubar},\\
\ol \nu^\perp &= \sinh(\theta) \nu + \cosh(\theta) \nu^\perp\nonumber,
\end{align}
for some function $\theta: \Sigma \to \R$.  Thus:
\begin{align*}
\ol \alpha(X) &= \langle\nabla_{X}^\perp \ol\nu, \ol\nu^\perp\rangle\\
 &=\langle \nabla_X^\perp\left(\cosh(\theta) \nu + \sinh(\theta) \nu^\perp\right), \sinh(\theta) \nu + \cosh(\theta) \nu^\perp\rangle\\
 &=-d\theta(X) + \alpha(X) .
 \end{align*}
In particular,
\begin{equation}
\label{eqn_dtheta}
\ol \alpha = \alpha - d\theta.
\end{equation}

\vspace{2mm}
\paragraph{\emph{Minimizing the $L^2$-norm of the covariant derivative}}
We consider the following functional on unit-length, outward-spacelike sections $\nu$ of $T^\perp \Sigma$:
\begin{equation}
\label{eqn_def_C}
\mathcal C(\nu) = \int_\Sigma \|\nabla^\perp \nu \|^2 dA,
\end{equation}
where
$$ \|\nabla^\perp \nu \|^2 = \sum_{i=1}^2\langle \nabla_{e_i} ^\perp \nu, \nu^\perp\rangle^2 = \sum_{i=1}^2 \alpha_\nu(e_i)^2 = |\alpha_\nu|^2$$
in any local orthonormal frame $\{e_1, e_2\}$ on $T\Sigma$.  We seek minimizers of $\mathcal{C}$ (``straight out'' directions), as they generalize the notion of parallel section.

\begin{proposition}
$\nu$ is a minimizer of $\mathcal{C}$ if and only if $\divm_\Sigma (\alpha_\nu) = 0$, where $\alpha_\nu$ is the connection 1-form associated to $\nu$.
\end{proposition}
\begin{proof}
Fix an outward-spacelike, unit length section $\nu_0$ of $T^\perp \Sigma$ and an arbitrary real valued function $\theta: \Sigma \to \R$.  We consider a variation
$$\nu_\epsilon = \cosh(\epsilon\theta) \nu + \sinh(\epsilon\theta) \nu^\perp.$$
By formula (\ref{eqn_dtheta}), the connection one-form $\alpha_\epsilon$ of $\nu_\epsilon$ is:
$$\alpha_\epsilon = \alpha_0 - \epsilon d\theta.$$
Then:
\begin{align*}
\left.\frac{d}{d\epsilon}\right|_{\epsilon=0} \mathcal{C}(\nu_\epsilon) &= \left.\frac{d}{d\epsilon}\right|_{\epsilon=0}\int_\Sigma |\alpha_\epsilon|^2 dA\\
&= -2\int_\Sigma \langle \alpha_0, d\theta\rangle dA\\
&= -2\int_\Sigma \langle d^*\alpha_0, \theta\rangle dA,
\end{align*}
where $d^*$ is the $L^2$-adjoint of $d$.
Thus $\nu_0$ is a critical point iff $-d^* \alpha_0 = \div_\Sigma (\alpha_0)  = 0$. Due to the convex nature of the functional $\mathcal{C}$, all critical points are minimizers.
\end{proof}

\begin{remark}
\label{rmk_conn_lap}
Alternatively, one may view $\nabla^\perp : \Gamma(T^\perp \Sigma) \to \Gamma(T^* \Sigma \otimes T^\perp \Sigma),$ and let $(\nabla^\perp)^* :  \Gamma(T^* \Sigma \otimes T^\perp \Sigma) \to \Gamma(T^\perp \Sigma)$ be its $L^2$-adjoint.  Direct computation shows
$$(\nabla^\perp)^* \nabla^ \perp \nu = -|\alpha_\nu|^2 \nu + \div_\Sigma (\alpha_\nu)\nu^\perp.$$
In particular, $\nu$ is a minimizer of $\mathcal{C}$ if and only if $\nu$ is an eigensection of the connection Laplacian $(\nabla^\perp)^* \nabla^ \perp$.
\end{remark}

\begin{proposition}
\label{prop_critical}
Minimizers of $\mathcal{C}$ exist, and any two differ by a hyperbolic rotation (\ref{eqn_nu_nubar}) by constant angle $\theta$.  
\end{proposition}

\begin{proof}
Fix some outward-spacelike unit normal $\nu$.  Suppose $\ol \nu$ is related to $\nu$ by (\ref{eqn_nu_nubar}). Let $\alpha = \alpha_\nu$ and $\ol \alpha = \alpha_{\ol \nu}$.
 By (\ref{eqn_dtheta}), $\ol \alpha$ is divergence-free iff $d^*d \theta = d^*\alpha$,
or, equivalently, 
\begin{equation}
\label{poisson}
\Delta_\Sigma \theta = \div_\Sigma (\alpha).
\end{equation}
By standard elliptic theory, a smooth solution $\theta$ to the Poisson equation (\ref{poisson}) exists (since $\int_\Sigma \div_\Sigma( \alpha) dA=0$). Moreover,
any two solutions differ by an element of $\ker \Delta_\Sigma$, which consists of the constant functions on $\Sigma$.
\end{proof}
Note the divergence of the connection 1-form appears in the definition of time flat surface and in the monotonicity formula (Theorem \ref{thm_tf_monotonicity}) for the Hawking mass.

\section{Variation of the Hawking mass}
\label{sec_monotonicity}
In this section we give a complete proof of Theorem \ref{thm_monotonicity}, a formula for the derivative of the Hawking mass (\ref{eqn_hawking}) along  a uniformly area expanding flow.  The equation  was first discovered by Malec, Mars, and Simon in \cite{MMS} (with similar calculations in \cite{imcf_spacetime}); we include the proof for clarity of exposition and to establish notation.  The reader may wish to postpone the proof of Theorem \ref{thm_monotonicity} to a second reading, skipping ahead to Corollary \ref{cor_monotonicity}.  Immediately after, we build on this formula in section \ref{sec_time_flat} to prove our main result, Theorem \ref{thm_tf_monotonicity}.  The main difference between the formulae in Theorems \ref{thm_tf_monotonicity} and \ref{thm_monotonicity} is that the former is in terms of the geometry of $\Sigma$, while the latter is in terms of the geometry of the swept-out hypersurface.

\begin{definition}
\label{def_uae}
A smooth family of embeddings $\Phi_\lambda : \Sigma \to N$ with initial flow velocity $\vec \xi = \left.\frac{\partial \Phi_\lambda}{\partial \lambda}\right|_{\lambda=0}$  is  \emph{uniformly area expanding} (at $\lambda =0$) if 
$$\left.\frac{\partial }{\partial \lambda}dA_\lambda\right|_{\lambda=0} = dA_0,$$
where $\Sigma_\lambda = \Phi_\lambda(\Sigma)$ and $dA_\lambda$ is the area form of $\Sigma_\lambda$.
By the first variation of area formula, this is equivalent to the condition
$$-\langle \vec \xi, \vec H\rangle=1,$$
where $\vec H$ is the mean curvature vector of $\Sigma_0$.
\end{definition}

Another interpretation of uniformly area expanding flows, adopted in \cite{imcf_spacetime}, is as follows.  Define the \emph{inverse mean curvature vector} of $\Sigma$:
$$\vec I = \frac{-\vec H}{\langle \vec H, \vec H \rangle},$$
assuming $\vec H$ is spacelike. 
Since $\vec I$ and $\vec I^\perp$ comprise a frame of $T^\perp \Sigma$, we can write
\begin{equation}
\label{eqn_gen_imcvf}
\vec \xi =  \vec I + \beta \vec I^\perp
\end{equation}
for some function  $\beta$. The coefficient of 1 on $\vec I$ is equivalent to 
$\vec \xi$ being uniformly area expanding.  Note that $\beta=\langle \vec \xi, \vec H^\perp\rangle$, and that $\vec \xi$ is spacelike if and only if $|\beta|<1$.

\begin{theorem} [Equation (11) of \cite{MMS}] 
\label{thm_monotonicity}
Let $\Sigma$ be a closed 2-manifold and $(N,\langle \cdot, \cdot\rangle)$  a time-oriented  Lorentzian spacetime.  For $\lambda \in [0, \epsilon)$, let $\Phi_\lambda : \Sigma \to N$ be a smooth family of spacelike embeddings, and set $\Sigma_\lambda = \Phi_\lambda(\Sigma)$.   Assume that mean curvature vector $\vec H$ of $\Sigma_0$ is spacelike, and that initial flow velocity $\vec \xi = \left.\frac{\partial \Phi_\lambda}{\partial \lambda}\right|_{\lambda=0}$ is outward-spacelike and uniformly
area expanding.  Then
\begin{align}
\left.\frac{d}{d\lambda} m_H(\Sigma_\lambda)\right|_{\lambda=0} \!\!&=
\sqrt{\frac{|\Sigma|}{(16\pi)^3}}  \Bigg\{4\pi(2-\chi(\Sigma))+\int_\Sigma\! \Bigg[16\pi( \mu - \beta J(\nu))+  |\mathring A|^2
+2\beta \langle \mathring A, \mathring p_\Sigma \rangle_\Sigma +|\mathring p_\Sigma|^2
 \nonumber\\
&\qquad+2\left(\frac{| \nabla^\Sigma H|^2}{H^2}+ 2\beta \ol p\left(\frac{\nabla^\Sigma H}{H}\right) + |\ol p|^2 \right)   - 2\beta \divm_\Sigma (\ol p) \,dA\Bigg]\Bigg\}.\label{eqn_mms}
\end{align}
\end{theorem}
To explain the notation:
\begin{compactitem}
\item $\Sigma_0$ is identified with $\Sigma$, $\chi(\Sigma)$ is its Euler characteristic, and $|\Sigma|$ is its area.
\item $\nu = \frac{\vec \xi}{|\vec \xi|}$ is the normalized flow velocity.
\item $\mu$ and $J$ are the energy and momentum density of $(M,g,k)$ (cf. the constraint equations (\ref{constraint1})--(\ref{constraint2})),
where $M= \bigcup_{\lambda \in [0,\epsilon)} \Sigma_\lambda$ is the spacelike hypersurface-with-boundary swept out by $\Sigma_\lambda$,  $g$ is its induced Riemannian metric, and $k$ is its second fundamental form in the future-timelike normal direction.
\item $p$ is defined to be $(\tr_M \!k)g - k$,  $p_\Sigma$ is the restriction of $p$ to $T\Sigma$,  and $\mathring p_\Sigma$ is the trace-free part of $p_\Sigma$. $\ol p$ is the 1-form on $\Sigma$ given by $\ol p(X) = p(X,\nu)$.
\item $A$ is the second fundamental form of $\Sigma$ as a submanifold of $M$ in the direction $\nu$ (cf. equation (\ref{eqn_k})), and $\mathring A$ is its trace-free part.
\item $\langle \cdot, \cdot \rangle_\Sigma$ and $|\cdot|^2$ are the tensor inner product and norm  on $\Sigma$ induced from $\langle \cdot, \cdot \rangle$.
\item $H = \tr_\Sigma A$ is the mean curvature of $\Sigma$ as a submanifold of $M$ in the direction $\nu$, and $\nabla^\Sigma $ is the gradient on $\Sigma$.
\end{compactitem}
Proofs of many intermediate formulae are deferred to the appendix.  Before giving the proof, we provide some preliminary details:
\vspace{2mm}
\paragraph{\emph{The swept-out hypersurface}} We will perform our calculations
with respect to the spacelike hypersurface $(M,g,k)$ swept out by the evolving
$\Sigma_\lambda$.  
We frame the normal bundle of $\Sigma$ by $\nu$ and $\nu^\perp$  and observe that $\nu$ is tangent to $M$
by definition.
The mean curvature vector of $\Sigma$ in $N$ decomposes
as
\begin{equation}
\label{eqn_vec_H_decomp}
-\vec H = H \nu - (\trSigmak) \nu^\perp,
\end{equation}
where $H$ is the  mean curvature of $\Sigma$ inside $M$ in the direction $\nu$,
and $\trSigmak = k(e_1,e_1) + k(e_2,e_2)$ in any local orthonormal frame $\{e_1, e_2\}$ of $T\Sigma$.  (This can be seen by tracing $\nabla^N - \nabla^\Sigma = (\nabla^N - \nabla^M) + (\nabla^M-\nabla^\Sigma)$ over $\Sigma$.)
Direct computation shows
\begin{equation}
\label{eqn_imcf_sweepout}
\vec \xi = -\frac{\nu}{\langle \nu, \vec H \rangle} = \frac{\nu}{H},
\end{equation}
which implies that the $\Sigma_\lambda$ \emph{evolve by inverse mean curvature flow in the hypersurface} $(M,g)$, a fact to be used in the proof of the theorem.  We also use this observation to derive an alternate expression for $\beta$.
Observe
$$-\vec H^\perp = -(\trSigmak) \nu + H \nu^\perp,$$
so that
\begin{align*}
\vec I + \beta \vec I^\perp &=  -\frac{\vec H+ \beta \vec H^\perp}{\langle \vec H, \vec H \rangle} \\
&=  \frac{(H-\beta \trSigmak)\nu +(\beta H-\trSigmak)\nu^\perp }{H^2 - (\trSigmak)^2} \\
&= \frac{\nu}{H}
\end{align*}
by setting 
\begin{equation}
\label{eqn_beta}
\beta = \frac{\trSigmak}{H}.
\end{equation}

\begin{proof}[Proof of Theorem \ref{thm_monotonicity}]
It will be useful to define a 2-tensor on $M$:
\begin{equation}
\label{eqn_p}
p= (\tr_M \!k)g - k,
\end{equation}
which satisfies:
\begin{equation}
\label{eqn_p_nu_nu}
p(\nu,\nu) = \trSigmak,
\end{equation}
a 1-form on $\Sigma$
$$\ol p(X) = p(X,\nu),$$
and a two-tensor $p_\Sigma$ on $\Sigma$ given by restricting $p$ to $T\Sigma$. 

We begin by using the decomposition (\ref{eqn_vec_H_decomp}) to rewrite the Hawking mass as:
\begin{align}
m_H(\Sigma) &= \sqrt{\frac{|\Sigma|}{16\pi}} \left[1 - \frac{1}{16\pi} \int_\Sigma\left( H^2 - (\trSigmak)^2\right) dA\right]\nonumber\\
 &= \sqrt{\frac{|\Sigma|}{16\pi}} \left[1 - \frac{1}{16\pi} \int_\Sigma \left( H^2 - p(\nu,\nu)^2\right) dA\right].\label{eqn_m_H_alt}
\end{align}
Thus, we require variational formulae for $|\Sigma|,$ $H^2$, $p(\nu,\nu)$, and $dA$ under IMCF in $M$.
\begin{lemma}
\label{lemma_variation}
Let $(M,g)$ be a Riemannian 3-manifold, and let $\Sigma_\lambda$ be a solution to inverse mean curvature flow  of 2-surfaces in $M$, say with  area forms $dA_\lambda$ and
mean curvatures $H_\lambda$ 
with respect to outward unit normals $\nu_\lambda$.  Let $\Sigma = \Sigma_0$.
Then:
\begin{align}
\label{eqn_var_dA}\left.\frac{\partial}{\partial \lambda} dA_\lambda \right|_{\lambda=0} &= dA\\
\label{eqn_var_A}\left.\frac{\partial |\Sigma_\lambda|}{\partial \lambda}  \right|_{\lambda=0} &= |\Sigma|\\
\label{eqn_var_H}\left.\frac{\partial H^2_\lambda}{\partial \lambda}  \right|_{\lambda=0} &= -2 H \Delta_\Sigma \left(\frac{1}{H} \right) - R + 2K - H^2 - |A|^2
\end{align}
where $\Delta_\Sigma$ is the Laplacian on $\Sigma$, $R$ is the scalar curvature of $g$, $K$ is the Gauss curvature of $\Sigma$, and $A$ is the second fundamental form
of $\Sigma$ in $M$ in the direction of $\nu=\nu_0$.  Here $dA$, $H$, etc. denote the area form, mean curvature, etc., of $\Sigma$.

Additionally, if $k$ is a symmetric $(0,2)$-tensor on $M$
and $p$ is defined by (\ref{eqn_p}), then
\begin{equation}
\label{eqn_var_p}\left.\frac{\partial (p(\nu_\lambda, \nu_\lambda)^2)}{\partial \lambda}  \right|_{\lambda=0} =\frac{2p(\nu,\nu)}{H} \left((\nabla^M_{\nu} p)(\nu,\nu) +2p\left(\nu, \frac{\nabla^\Sigma H}{H}\right)\right),
\end{equation}
where $\nabla^M$ is the Levi-Civita connection on $M$.
\end{lemma}
This is proved in the appendix.

Now we differentiate (\ref{eqn_m_H_alt}) in $\lambda$ and use Lemma \ref{lemma_variation}:
\begin{align*}
\left.\frac{d m_H(\Sigma_\lambda)}{d\lambda}\right|_{\lambda=0} &= \frac{1}{2} \sqrt{\frac{|\Sigma|}{16\pi}} \left[1 - \frac{1}{16\pi} \int_\Sigma \left(H^2 - p(\nu,\nu)^2 \right) dA\right]\\
&\qquad+  \sqrt{\frac{|\Sigma|}{16\pi}} \left[- \frac{1}{16\pi} \int_\Sigma\left. \frac{\partial}{\partial \lambda}\left( H^2 - p(\nu,\nu)^2\right) dA\right|_{\lambda=0}\right]\\
&= \sqrt{\frac{|\Sigma|}{(16\pi)^3}}  \Bigg\{8\pi - \int_\Sigma \Bigg[\frac{1}{2} H^2-\frac{1}{2}p(\nu,\nu)^2-2 H \Delta_\Sigma \left(\frac{1}{H} \right) - R + 2K - H^2 - |A|^2\\
&\qquad-\frac{2p(\nu,\nu)}{H}\left((\nabla^M_{\nu} p)(\nu,\nu) + 2p\left(\nu, \frac{\nabla^\Sigma H}{H}\right)\right)
+ H^2 - p(\nu,\nu)^2\Bigg] dA\Bigg\}.
\end{align*}

The next step is to relate $\nabla^M_\nu p$ to the divergence of $p$, which will eventually allow the use of the constraint equations.  For this, we use
\begin{equation}
(\div_M (p))(\nu) = (\nabla^M_\nu p)(\nu,\nu) + \div_\Sigma (\ol p) + H p(\nu,\nu) - \langle A, p_\Sigma\rangle_\Sigma \label{identity_div_M},
\end{equation}
which is proved in the appendix. Here, $\langle \cdot, \cdot \rangle_\Sigma$ is the restriction of $\langle \cdot, \cdot \rangle$ to $T\Sigma$.  In the appendix, we also use the constraint equations to show:
\begin{align}
R &= 16\pi \mu  +|p_\Sigma|_\Sigma^2 - \frac{1}{2} ( \tr_\Sigma p_\Sigma)^2+ 2|\ol p|^2 +\frac{1}{2} p(\nu,\nu)^2 - p(\nu,\nu)\tr_\Sigma p_\Sigma, \label{identity_R}\\
 -2 \div_M (p) &= 16\pi J.  \label{identity_J}
\end{align}
Below, we use the Gauss--Bonnet theorem, integration by parts, identities (\ref{identity_div_M})--(\ref{identity_J}), and the fact that $\beta = \frac{\trSigmak}{H}=\frac{p(\nu,\nu)}{H}$ (by (\ref{eqn_beta}) and (\ref{eqn_p_nu_nu})).
\begin{align*}
\left.\frac{d m_H(\Sigma_\lambda)}{d\lambda}\right|_{\lambda=0}&= \sqrt{\frac{|\Sigma|}{(16\pi)^3}}  \left\{4\pi(2-\chi(\Sigma))+\int_\Sigma\left[- \frac{1}{2} H^2+\frac{3}{2}p(\nu,\nu)^2+\frac{2| \nabla^\Sigma H|^2}{H^2}  + |A|^2 + R\right.\right.\\
&\;\;\left.\left.+\frac{2p(\nu,\nu)}{H}\left(\!(\div_M (p))(\nu) - \div_\Sigma (\ol p) - H p(\nu,\nu) + \langle A, p_\Sigma\rangle_\Sigma+ 2\ol p\left(\frac{\nabla^\Sigma H}{H}\right)
\!\!\right)\right]\!dA
\right\}\\
&= \sqrt{\frac{|\Sigma|}{(16\pi)^3}}  \left\{4\pi(2-\chi(\Sigma))+\int_\Sigma\left[- \frac{1}{2} H^2+\frac{3}{2}p(\nu,\nu)^2+\frac{2| \nabla^\Sigma H|^2}{H^2}  + |A|^2 \right.\right.\\
&\qquad + 16\pi \mu  +|p_\Sigma|_\Sigma^2 - \frac{1}{2} ( \tr_\Sigma p_\Sigma)^2+ 2|\ol p|^2 +\frac{1}{2} p(\nu,\nu)^2 - p(\nu,\nu)\tr_\Sigma p_\Sigma -16\pi\beta J(\nu)\\
&\qquad\left.\left.+\frac{2p(\nu,\nu)}{H}\left( - \div_\Sigma (\ol p) - H p(\nu,\nu) + \langle A, p_\Sigma\rangle_\Sigma +2\ol p\left(\frac{\nabla^\Sigma H}{H}\right)
\right)\right]dA
\right\}\\
&= \sqrt{\frac{|\Sigma|}{(16\pi)^3}}  \left\{4\pi(2-\chi(\Sigma))+\int_\Sigma\left[- \frac{1}{2} H^2+\frac{2|\nabla^\Sigma H|^2}{H^2}  + |A|^2 \right.\right.\\
&\qquad + 16\pi( \mu - \beta J(\nu))  +|p_\Sigma|_\Sigma^2 - \frac{1}{2} ( \tr_\Sigma p_\Sigma)^2+ 2|\ol p|^2  -\beta H \tr_\Sigma p_\Sigma \\
&\qquad\left.\left. - 2\beta \div_\Sigma (\ol p) + 2\beta \langle A, p_\Sigma\rangle_\Sigma + 4\beta \ol p\left(\frac{\nabla^\Sigma H}{H}\right)\right]dA\right\}.
\end{align*}
It is convenient to work with the traceless parts of $p_\Sigma$ and $A$:
\begin{align*}
\mathring p_\Sigma &= p_\Sigma - \frac{1}{2}\left( \tr_\Sigma p_\Sigma\right) g_\Sigma,\\
\mathring A &= A - \frac{1}{2} H g_\Sigma,
\end{align*}
where $g_\Sigma$ is the restriction of $g$ to $T\Sigma$.  Elementary computations show:
\begin{align*}
|\mathring p_\Sigma|^2 &= |p_\Sigma|^2 -\frac{1}{2} \left(\tr_\Sigma p_\Sigma\right)^2,\\
|\mathring A|^2 &= |A|^2 - \frac{1}{2}H^2,\\
\langle \mathring p_\Sigma, \mathring A\rangle_\Sigma &= \langle p_\Sigma, A \rangle_\Sigma - \frac{1}{2} H \tr_\Sigma p_\Sigma.
\end{align*}

Then we have:
\begin{align*}
\left.\frac{d m_H(\Sigma_\lambda)}{d\lambda}\right|_{\lambda=0} &= \sqrt{\frac{|\Sigma|}{(16\pi)^3}}  \left\{4\pi(2-\chi(\Sigma))+\int_\Sigma \bigg(16\pi( \mu - \beta J(\nu))+\left[|\mathring A|^2
+2\beta \langle \mathring p_\Sigma, \mathring A \rangle_\Sigma +|\mathring p_\Sigma|^2 \right]
 \right.\\
&\qquad  \left.\left.+2\left[\frac{| \nabla^\Sigma H|^2}{H^2} + 2\beta \ol p\left( \frac{\nabla^\Sigma H}{H}\right) + |\ol p|^2 \right]   - 2\beta \div_\Sigma (\ol p)\,\right) dA\right\}.
 \end{align*}
which is (\ref{eqn_mms}).
\end{proof}

Malec, Mars, and Simon observed formula (\ref{eqn_mms}) implies a monotonicity result for the Hawking mass:

\begin{corollary}
\label{cor_monotonicity}
Under the hypotheses of Theorem \ref{thm_monotonicity},  if $\Sigma$ is connected, if $N$ satisfies the dominant energy condition, and if $\divm_\Sigma (\ol p) = 0$ then 
$$\left.\frac{d}{d\lambda} m_H(\Sigma_\lambda)\right|_{\lambda=0} \geq 0.$$
\end{corollary}
\begin{proof}
Consider the last equation in the above proof.  For $\Sigma$ connected, $\chi(\Sigma) \leq 2$.
The fact that $\vec \xi$ is spacelike implies $|\beta| < 1$, which in turn implies pointwise $\geq 0$ inequalities for the terms in square brackets.
Next, recall from (\ref{DEC}) that the dominant energy condition on $N$ implies that $\mu \geq |J(\nu)| \geq \beta J(\nu)$. 
Thus,  if $\div_\Sigma (\ol p)$ vanishes, then $\left.\frac{d m_H(\Sigma_\lambda)}{d\lambda}\right|_{\lambda=0} \geq 0$.  
\end{proof}

This motivates a flow of spacelike 2-surfaces in $N$ for which the Hawking mass is nondecreasing: require $\{\Sigma_\lambda\}$ to be uniformly area expanding and $\ol p$ to be divergence-free on each $\Sigma_\lambda$.  This was first proposed in \cite{MMS}.
It is not immediately obvious that $\ol p$ is defined without referring to the $\emph{a priori}$ swept-out hypersurface $M$.  However, a simple computation in the appendix shows that $\ol p = \alpha_\nu$, the connection 1-form of $\Sigma$ in the direction $\nu=\frac{\vec \xi}{|\vec \xi|}$, which depends only on $\Sigma$ and its normal bundle in $N$.  Generally, if such a flow is constrained to live inside a fixed spacelike hypersurface $M$, it is not possible to require $\div_\Sigma(\ol p)$ to vanish on each $\Sigma_\lambda$.  Asking $\ol p = \alpha_\nu$ to be divergence-free means asking $\nu$ to be a straight out direction (i.e., a minimizer of $\mathcal{C}$) in the sense of section \ref{sec_connection}.  Generally, such a direction is not  tangent to $M$.  

Thus, it is more appropriate to take the ``invariant'' approach of \cite{MMS},  allowing $\Sigma_\lambda$ to evolve within $N$.  Such a flow may have a reasonable existence theory, but it is not clear that the $\Sigma_\lambda$ have good asymptotics at infinity --- which is essential for detecting the total mass or energy of $N$.  For example, let $N=\R^{3,1}$ (the Minkowski spacetime), and let $\Sigma$ be a closed surface in $N$ with positive Hawking mass (see Example \ref{ex_sphere}).  For a flow beginning at $\Sigma$ for which the Hawking mass is nondecreasing, something must go wrong, since the total mass and energy of $\R^{3,1}$ are zero.

Another flow discussed in \cites{MMS, imcf_spacetime} is to require $\beta=\beta(\lambda)$ to be constant on each $\Sigma_\lambda$ and between $-1$ and $1$.  In that case, the problematic term $-2\beta \div_{\Sigma_\lambda} (\ol p)$ integrates over $\Sigma_\lambda$ to zero, which leads to monotonicity of the Hawking mass.  However, it was proved in \cite{imcf_spacetime} that such a flow is always forwards-backwards parabolic (and so solutions beginning with general initial data do not exist).

\section{Time flat surfaces}
\label{sec_time_flat}
In this section we provide examples of time flat surfaces and derive  a variation formula for the Hawking mass (Theorem \ref{thm_tf_monotonicity}) that is well-suited to studying them.

Recall Definition \ref{def_TF} of a time flat surface from the introduction.  Based on section \ref{sec_connection}, an equivalent condition is that
$\nu_H := -\frac{\vec H}{|\vec H|}$ is a minimizer of $\mathcal{C}$ (i.e., a ``straight out'' direction), or, in light of Remark \ref{rmk_conn_lap},
 that $\nu_H$ is an eigenfunction of the Laplacian $(\nabla^\perp)^* \nabla^\perp$.

Examples of time flat surfaces include strictly mean-convex surfaces contained in a spacelike hyperplane in $\R^{3,1}$.  More generally, 
any strictly mean-convex surface contained in $t=$ constant slice of a static spacetime 
\begin{equation}
\label{eqn_static}
\langle \cdot, \cdot \rangle = -u^2 dt^2 +g
\end{equation}
is time flat.  Here, $(M,g)$ is a fixed Riemannian 3-manifold and $u>0$ is a function on $M$. More generally still, a strictly mean-convex
surface in a totally geodesic spacelike hypersurface $M$ of a spacetime is time flat. Next, any spherically symmetric sphere
(with spacelike mean curvature vector) in a spherically symmetric spacetime is time flat. Finally, any surface 
whose mean curvature vector is spacelike and parallel with respect to $\nabla^\perp$ is time flat.

The motivation for the time flat condition arises from several considerations:
\begin{itemize}
\item What is a canonical choice of outward-spacelike unit normal to $\Sigma$?  By analogy with the Frenet frame for curves in 3-space, 
$\nu_H =-\frac{\vec H}{|\vec H|}$ is a natural choice (and ${\nu_H}^\perp$ would play the role of the binormal).  On the other hand, minimizers of $\mathcal{C}$ discussed in section 
\ref{sec_connection} (the so-called ``straight out'' directions) are also geometrically natural.  A time flat surface is one for which these notions coincide.
\item Consider the problem of flowing $\Sigma$ inside $N$ with a uniformly area expanding velocity.  Generally, the variation of the Hawking mass along the flow could have any sign.  However, if $\Sigma$ is time flat, then \emph{any} such flow leads to a nonnegative variation of the Hawking mass (Corollary \ref{cor_tf}).
\item The Wang--Yau quasi-local mass of  $\Sigma$ is defined by optimizing certain isometric embeddings of $\Sigma$ into $\R^{3,1}$ \cites{wang_yau1, wang_yau2}.  If $\Sigma$ is time flat, then its natural embedding into a $t= $ constant slice of the Minkowski spacetime is optimal \cite{cww}.
\end{itemize}
Further motivation is given in section \ref{sec_discussion}.

We now prove Theorem \ref{thm_tf_monotonicity}, stated in the introduction, which illustrates how time flat surfaces relate to uniformly area expanding flows and the Hawking mass.  The essential difference between the formulas in Theorems \ref{thm_tf_monotonicity} and \ref{thm_monotonicity} is that the former is expressed in terms of the  geometry of $\Sigma$, while the latter is expressed in terms of the geometry of the swept-out hypersurface.

\begin{proof}[Proof of Theorem \ref{thm_tf_monotonicity}]
The idea is to begin with (\ref{eqn_mms}) and rewrite all quantities in the integral in terms of the $\nu_H$ and $\nu_H^\perp$ directions, without referring to the swept-out hypersurface $M$.   First, we give some preliminaries: define $\vec \II$ to be the second fundamental form of $\Sigma$ in $N$ and $\mathring{\vec \II}$ its traceless part.

Using the definitions of $\nu, \vec \xi, \nu_H,$ and $\beta$, we have
\begin{align}
\label{eqn_nu_nu_H}
\nu &= \frac{1}{\sqrt{1-\beta^2}} \nu_H + \frac{\beta}{\sqrt{1-\beta^2}} \nu_H^\perp\\
&= \cosh(\theta) \nu_H + \sinh(\theta) \nu_H^\perp \label{eqn_nu_theta},
\end{align}
for the function $\theta = \tanh^{-1}(\beta)$.

\vspace{2mm}
\paragraph{\emph{The $\mu$ and $J$ terms}} By the definition of the the energy and momentum densities, Einstein's equation, and (\ref{eqn_nu_nu_H}), we have
\begin{align*}
16\pi (\mu - \beta J(\nu)) &= 2G (\nu^\perp, \nu^\perp) - 2\beta G(\nu^\perp,\nu)\\
&= 2G(\nu^\perp, \nu^\perp-\beta\nu)\\
&=\frac{2}{1-\beta^2} G(\beta \nu_H + \nu_H^\perp, \beta \nu_H + \nu_H^\perp - \beta \nu_H -\beta^2 \nu_H^\perp)\\
&=2G(\nu_H^\perp, \beta \nu_H + \nu_H^\perp).
\end{align*}
Using the definition of $\nu_H$ and $\vec \xi$ and integrating, we have:
\begin{equation}
\label{eqn_G}
\int_\Sigma 16\pi (\mu - \beta J(\nu)) dA=\int_\Sigma 2G(-\vec H^\perp, \vec \xi^\perp)dA. 
\end{equation}

\vspace{2mm}
\paragraph{\emph{The $\mathring A$ and $\mathring p_\Sigma$ terms}}  Presently, we rewrite the terms:
\begin{equation}
\label{first_three_start}
\int_\Sigma \left(
|\mathring A|^2+2\beta \langle  \mathring A, \mathring p_\Sigma \rangle_\Sigma +|\mathring p_\Sigma|^2\right)\!dA.
\end{equation}
The first observation is that $A$ can be expressed without referring to $M$.  For $e_i, e_j$ tangent to $\Sigma$,
\begin{align*}
A(e_i,e_j) &= -\langle \nabla^M_{e_i} e_j, \nu\rangle\\
&= -\langle \nabla^N_{e_i} e_j, \nu\rangle\\
&= -\langle \vec \II (e_i,e_j), \nu\rangle,
\end{align*}
since $\nabla^N_{e_i} e_j - \nabla^M_{e_i}e_j$ is orthogonal to $\nu$. 
Taking the traceless parts, we have:
\begin{equation}
\label{eqn_A}
\mathring A(e_i,e_j) =-\langle\mathring{\vec  \II}(e_i,e_j), \nu\rangle.
\end{equation}

The second observation is that $\mathring p_\Sigma$ has a nicer expression.  By definition of $p_\Sigma$:
$$p_\Sigma (e_i,e_j) = (\tr_M k) g(e_i,e_j) - k(e_i,e_j)$$
Now, by the definition of $k$:
\begin{align*}
k(e_i,e_j) &= - \langle \nabla^N_{e_i}e_j, \nu^\perp\rangle\\
&= - \langle \vec\II(e_i,e_j), \nu^\perp\rangle.
\end{align*}
In particular,
\begin{equation}
\label{eqn_p_Sigma}
\mathring p_\Sigma(e_i,e_j) =  \langle \mathring {\vec \II}(e_i,e_j), \nu^\perp\rangle.
\end{equation}

We express formulas (\ref{eqn_A}) and (\ref{eqn_p_Sigma}) for $\mathring A$ and $\mathring p_\Sigma$ in terms of $\nu_H$ and $\nu_H^\perp$.
Define the scalar-valued second fundamental forms of $\Sigma$ in the directions $\nu_H$ and $\nu_H^\perp$:
\begin{align*}
\II_{\nu_H} &:= -\langle \vec \II, \nu_H\rangle\\
\II_{\nu_H^\perp} &:= -\langle \vec \II, \nu_H^\perp\rangle,
\end{align*}
and let $\mathring \II_{\nu_H}$ and $\mathring \II_{\nu_H^\perp}$ be their traceless parts.

Thus:
\begin{align*}
\mathring A &=-\langle\mathring  {\vec \II},  \frac{1}{\sqrt{1-\beta^2}} \nu_H + \frac{\beta}{\sqrt{1-\beta^2}} \nu_H^\perp\rangle\\
&= \frac{1}{\sqrt{1-\beta^2}}\mathring \II_{\nu_H}  + \frac{\beta}{\sqrt{1-\beta^2}}\mathring \II_{\nu_H^\perp},
\end{align*}
and:
\begin{align*}
\mathring p_\Sigma &=\langle\mathring {\vec \II},  \frac{\beta}{\sqrt{1-\beta^2}} \nu_H + \frac{1}{\sqrt{1-\beta^2}} \nu_H^\perp\rangle\\
&= -\frac{\beta}{\sqrt{1-\beta^2}}\mathring \II_{\nu_H}  - \frac{1}{\sqrt{1-\beta^2}}\mathring \II_{\nu_H^\perp}.
\end{align*}

Now, we compute the three terms in (\ref{first_three_start}):
\begin{align*}
|\mathring A|^2 &= \frac{1}{1-\beta^2} |\mathring \II_{\nu_H} |^2 +\frac{2\beta}{1-\beta^2} \langle \mathring \II_{\nu_H} ,\mathring \II_{\nu_H^\perp}\rangle +  \frac{\beta^2}{1-\beta^2} |\mathring \II_{\nu_H^\perp} |^2
\end{align*}

\begin{align*}
2\beta \langle \mathring A, \mathring p_\Sigma\rangle &= 2\beta \left(\frac{-\beta}{1-\beta^2}  |\mathring \II_{\nu_H}|^2 + \frac{-1-\beta^2}{1-\beta^2} \langle \mathring \II_{\nu_H}, \mathring \II_{\nu_H^\perp}\rangle + \frac{-\beta}{1-\beta^2} |\mathring \II_{\nu_H^\perp}|^2   \right)
\end{align*}

\begin{align*}
|\mathring p_\Sigma|^2 &= \frac{\beta^2}{1-\beta^2} |\mathring \II_{\nu_H} |^2 +\frac{2\beta}{1-\beta^2} \langle \mathring \II_{\nu_H} ,\mathring \II_{\nu_H^\perp}\rangle +  \frac{1}{1-\beta^2} |\mathring \II_{\nu_H^\perp} |^2.
\end{align*}

Summing the last three lines and integrating, we have 
\begin{equation}
\label{first_three_finish}
\int_\Sigma \left( |\mathring A|^2 + 2\beta\langle \mathring A, \mathring p_\Sigma \rangle + |\mathring p_\Sigma|^2\right)dA = \int_\Sigma \left( |\mathring \II_{\nu_H} |^2 + 2\beta \langle  \mathring \II_{\nu_H}, \mathring \II_{\nu_H^\perp}\rangle + |\mathring \II_{\nu_H^\perp}|^2\right)dA.
\end{equation}

\vspace{2mm}
\paragraph{\emph{The $H$ and $\ol p$ terms}} In this next part of the proof, we rewrite:
\begin{equation}
\label{eqn_next_three_start}
2 \int_\Sigma \left(\frac{| \nabla^\Sigma H|^2}{H^2} +2\beta \ol p\left( \frac{\nabla^\Sigma H}{H}\right) + |\ol p|^2\! \right)dA.
\end{equation}

We start with the relationship between $H$ and $|\vec H| := \sqrt{\langle \vec H, \vec H\rangle}$.  By (\ref{eqn_vec_H_decomp}),
\begin{align*}
H &= -\langle \vec H, \nu \rangle\\
 &= -\langle \vec H, \frac{1}{\sqrt{1-\beta^2}} \nu_H + \frac{\beta}{\sqrt{1-\beta^2}} \nu_H^\perp\rangle\\
&= \frac{|\vec H|}{\sqrt{1-\beta^2}}.
\end{align*}
Logarithmic differentiation shows:
$$\frac{\nabla^\Sigma H}{H} = \frac{\nabla^\Sigma |\vec H|}{|\vec H|} + \frac{\beta \nabla^\Sigma \beta}{1-\beta^2}.$$

In the appendix, we prove:
\begin{equation}
\ol p = \alpha, \label{identity_p_omega}
\end{equation}
where $\alpha$ is the connection 1-form associated to $\nu = \frac{\vec \xi}{|\vec \xi|}$ (cf. equation (\ref{eqn_conn_form})). By equation (\ref{eqn_dtheta}), $\alpha$
and $\alpha_H$ are related by
\begin{align*}
\alpha
&= \alpha_H - d\theta\\
&= \alpha_H -\frac{d\beta}{1-\beta^2}.
\end{align*}
since $\theta = \tanh^{-1}(\beta)$. We now compute the terms in (\ref{eqn_next_three_start}):
\begin{align*}
\left| \frac{\nabla^\Sigma H}{H}\right|^2 &= \left| \frac{\nabla^\Sigma |\vec H|}{|\vec H|}\right|^2 + \frac{2\beta\langle \nabla^\Sigma |\vec H|, \nabla^\Sigma \beta \rangle}{|\vec H|(1-\beta^2)} + \frac{\beta^2|\nabla^\Sigma \beta|^2}{(1-\beta^2)^2}\\
2\beta \ol p\left( \frac{\nabla^\Sigma H}{H}\right) &= 2\beta \left(\alpha_H -\frac{d\beta}{1-\beta^2}\right)\left( \frac{\nabla^\Sigma |\vec H|}{|\vec H|} + \frac{\beta \nabla^\Sigma \beta}{1-\beta^2}\right)\\
&=2\beta\left(\alpha_H\left(\frac{\nabla^\Sigma |\vec H|}{|\vec H|}\right)    + \frac{\beta}{1-\beta^2}\alpha_H(\nabla^\Sigma \beta)
-\frac{\langle \nabla^\Sigma \beta, \nabla^\Sigma |\vec H|\rangle}{|\vec H|(1-\beta^2)} -\frac{\beta |\nabla^\Sigma \beta|^2}{(1-\beta^2)^2}\right)\\
|\ol p|^2 &= |\alpha_H|^2 -\frac{2}{1-\beta^2}\alpha_H(\nabla^\Sigma \beta) + \frac{ |\nabla^\Sigma \beta|^2}{(1-\beta^2)^2}.
\end{align*}

Adding the last three formulae produces:
\begin{align*}
\frac{| \nabla^\Sigma H|^2}{H^2} +2\beta \ol p\left( \frac{\nabla^\Sigma H}{H}\right) + |\ol p|^2 &=  \left| \frac{\nabla^\Sigma |\vec H|}{|\vec H|}\right|^2 + 2\beta\alpha_H\left(\frac{\nabla^\Sigma |\vec H|}{|\vec H|}\right) + |\alpha_H|^2\\
&\qquad -2\alpha_H\left(\nabla^\Sigma \beta\right)+ \frac{|\nabla^\Sigma \beta|^2}{1-\beta^2}.
\end{align*}
Multiplying by 2, integrating (including integrating the second to last term above by parts), we have computed (\ref{eqn_next_three_start}) as:
\begin{align}
\label{eqn_next_three_finish}
2\int_\Sigma \left( \frac{| \nabla^\Sigma H|^2}{H^2} +2\beta \ol p\left( \frac{\nabla^\Sigma H}{H}\right) + |\ol p|^2\right)dA &= 2 \int_\Sigma  \Bigg(\left| \frac{\nabla^\Sigma |\vec H|}{|\vec H|}\right|^2 + 2\beta\alpha_H\left(\frac{\nabla^\Sigma |\vec H|}{|\vec H|}\right) +|\alpha_H|^2\nonumber\\
&\qquad + 2\beta \div_\Sigma(\alpha_H) + \frac{|\nabla^\Sigma \beta|^2}{1-\beta^2} \Bigg)dA.
\end{align}

\vspace{2mm}
\paragraph{\emph{The $- 2\beta \divm_\Sigma (\ol p)$ term}}
We now compute the last term in formula (\ref{eqn_mms}).  Recalling $\ol p =\alpha$, we have:

 \begin{align}
-2\int_\Sigma \beta \div_\Sigma(\alpha) dA &=-2\int_\Sigma \beta \div_\Sigma(\alpha_H - d\theta) dA \nonumber\\
&=-2\int_\Sigma \beta (\div_\Sigma(\alpha_H) - \Delta \theta) dA\nonumber\\
&=-2\int_\Sigma \left(\beta \div_\Sigma(\alpha_H) +\langle \nabla^\Sigma \beta, \nabla^\Sigma \theta\rangle\right)\! dA\nonumber\\
&=-2\int_\Sigma\left( \beta \div_\Sigma (\alpha_H) + \frac{|\nabla^\Sigma \beta|^2}{1-\beta^2}  \right)\!dA, \label{eqn_tf_lb2}
\end{align}
having integrated by parts and used $\nabla^\Sigma \theta = \frac{\nabla^\Sigma \beta}{1-\beta^2}$.

\vspace{2mm}
\paragraph{\emph{Conclusion}} Taking formula (\ref{eqn_mms}) and substituting in (\ref{eqn_G}), (\ref{first_three_finish}), (\ref{eqn_next_three_finish}), and (\ref{eqn_tf_lb2}), and taking cancellations, we have:

\begin{align*}
\left.\frac{d}{d\lambda} m_H(\Sigma_\lambda)\right|_{\lambda=0} \!\!&=
\sqrt{\frac{|\Sigma|}{(16\pi)^3}}  \Bigg\{4\pi(2-\chi(\Sigma))+\int_\Sigma\! \Bigg[2G(-\vec H^\perp, \vec \xi^\perp)+    |\mathring \II_{\nu_H} |^2 + 2\beta \langle  \mathring \II_{\nu_H}, \mathring \II_{\nu_H^\perp}\rangle + |\mathring \II_{\nu_H^\perp}|^2 
 \nonumber\\
&\qquad +2\left(\left| \frac{\nabla^\Sigma |\vec H|}{|\vec H|}\right|^2 + 2\beta\alpha_H\left(\frac{\nabla^\Sigma |\vec H|}{|\vec H|}\right) +|\alpha_H|^2 \right)   + 2\beta \div_\Sigma (\alpha_H) \Bigg]dA\Bigg\}.
\end{align*}
This completes the proof of Theorem \ref{thm_tf_monotonicity}.
\end{proof}

Now Corollary \ref{cor_tf} from the introduction can be proved by gathering the following facts: $\chi(\Sigma) \leq 2$ for $\Sigma$ connected; $G(-\vec H^\perp, \vec \xi^\perp) \geq 0$ by the dominant energy condition (since $-\vec H^\perp$ and $\vec \xi^\perp$ are future-timelike); $\div_\Sigma(\alpha_H)=0$ by time flatness; the remaining two groups of three terms are pointwise nonnegative since $|\beta| < 1$.

\begin{example}
\label{ex_sphere}
Suppose $\Sigma$ is a unit sphere contained in the $t=0$ slice of the Minkowski spacetime (which is a time flat surface).  In the above formula, $\chi(\Sigma)=2$, $G=0$ since this spacetime is vacuum, $\mathring \II_{\nu_H}=0$ because $\Sigma$ is totally umbilic in the $t=0$ slice, $\mathring \II_{\nu_H^\perp}=0$ because the $t=0$ slice is totally geodesic, $|\vec H|=2$ is constant, and $\alpha_H=0$.  Thus, for any uniformly area expanding flow $\Sigma_\lambda$ out of $\Sigma$, the Hawking mass does not change to first order.  However, we remark that there exist such flows for which $\left.\frac{d^2 m_H(\Sigma_\lambda)}{d\lambda^2}\right|_{\lambda=0}$ is strictly positive.
\end{example}

The following conjecture would justify the terminology of ``time flat'': such surfaces in the Minkowski spacetime ought not to have timelike oscillations.
\begin{conjecture}
\label{conj_time_flat}
Let $\Sigma$ be a closed surface in $\R^{3,1}$ that is contained in a complete, spacelike hypersurface.   If $\Sigma$ is time flat, then $\Sigma$ is contained in a spacelike hyperplane. 
\end{conjecture}
In particular, $\Sigma$ must have nonpositive Hawking mass, since this is a well-known fact for surfaces in $\R^3$.

\begin{remark}

As first evidence for its validity, Po-Ning Chen, Mu-Tao Wang, and Ye-Kai Wang have proved 
Conjecture \ref{conj_time_flat} for two separate special cases: a) if $\Sigma$ has vanishing $\alpha_H$, or b) $\Sigma$ is axially symmetric, with positive Gauss curvature, and can be written as the graph over a convex surface.  They also showed an infinitesimal rigidity statement: that a time flat surface contained in a spacelike hyperplane of $\R^{3,1}$ cannot be perturbed, on the infinitesimal level, to a time flat surface in a nontrivial way \cite{cww}.
\end{remark}

We conclude this section by considering a lower-dimensional analogy.
Consider a simple closed, oriented curve $\gamma$ in a (2+1)-dimensional Lorentzian manifold.  Assume that $\gamma$ is spacelike with spacelike, non-vanishing mean curvature vector.  In particular, the Frenet frame $\{T,N,B\}$ of $\gamma$ is everywhere defined.  The connection 1-form on the normal bundle of $\gamma$ with respect the frame $\{N,B\}$ is  $\tau ds$, where $\tau$ is the torsion and $ds$ is the volume form on $\gamma$.  In particular, if Definition \ref{def_TF} is extended to dimension 1, time flatness would be equivalent to  $\tau ds$ being divergence-free, i.e. the torsion being constant.  This provides motivation for considering closed curves of constant torsion in Minkowski 3-space $\R^{2,1}$.  Closed curves of constant nonzero torsion and non-vanishing curvature in Euclidean space $\R^3$ are known to exist \cite{weiner}.  Such curves in $\R^{2,1}$ that are spacelike with spacelike mean curvature also appear to exist, though this has not appeared in the literature.  Based on this analogy, we anticipate that nontrivial, closed time flat surfaces in $\R^{3,1}$ exist, although Conjecture \ref{conj_time_flat} would constrain them.

In a forthcoming paper, the authors prove the (2+1)-dimensional version of Conjecture \ref{conj_time_flat} \cite{const_torsion} (see also \cite{nonnegative_torsion}).  Define a spacelike curve in $\R^{2,1}$ to be time flat if it has spacelike mean curvature vector and constant torsion.

\begin{theorem}
Let $\gamma$ be a simple, closed curve in $\R^{2,1}$ that is contained in a complete, spacelike hypersurface.  If $\gamma$ is time flat, then $\gamma$ has zero torsion and is hence contained in a spacelike hyperplane.
\end{theorem}
The proof is nontrivial because nonplanar, simple, closed, time flat  curves in $\R^{2,1}$ do exist.  It is therefore essential to use the hypothesis that $\gamma$ is contained in a spacelike hypersurface.

\section{Discussion: flows of time flat surfaces}
\label{sec_discussion}
In this section, we discuss how time flat surfaces might be useful in addressing a range of questions surrounding the concept of mass in general relativity.
Consider a closed, spacelike surface $\Sigma$ in an asymptotically flat spacetime $N$.  For problems pertaining to quasi-local mass and the Penrose inequality, a flow beginning at $\Sigma$ and possessing the following properties would be desirable:
\begin{itemize}
\item The flow has a (weak) existence theory, which possibly allows for discrete jumps in the surfaces.
\item The flowed surfaces sweep out a spacelike hypersurface (modulo jumps) with good asymptotics at infinity.
\item A nontrivial quantity (such as the Hawking mass) is nondecreasing along the flow (including at the jumps), and this quantity limits to the ADM mass or ADM energy at infinity.
\end{itemize}

One immediate issue with this picture is the following: there exist spacelike perturbations of a round sphere in a $t=$ constant slice of $\R^{3,1}$ that have spacelike mean curvature
vector and \emph{positive} Hawking mass.  Yet, by any reasonable definition, the mass and energy of $\R^{3,1}$ vanish.  However, such surfaces must have timelike oscillations,
which ought to prevent them from being time flat (Conjecture \ref{conj_time_flat}).  This provides yet another reason for considering the time flat condition and motivates the following flow:

\begin{definition}
A smooth family of surfaces $\{\Sigma_\lambda\}$  in $N$ is said to satisfy \emph{uniformly area expanding time flat flow} if
each $\Sigma_\lambda$ is time flat, and the flow velocity is uniformly area expanding (cf. Definition \ref{def_uae}).
\end{definition}
As an immediate consequence of Corollary \ref{cor_tf} from the introduction, we have:
\begin{corollary}
Suppose $\{\Sigma_\lambda\}$ are  closed and connected and satisfy uniformly area expanding time flat flow with outward-spacelike velocity, and that $N$ obeys the dominant energy condition. Then 
$$\frac{d}{d\lambda} m_H(\Sigma_\lambda) \geq 0.$$
\end{corollary}

We give examples of such flows: first, any smooth inverse mean curvature flow of mean-convex surfaces in a $t= $ constant slice of a static spacetime (\ref{eqn_static}) satisfies these conditions.  Alternatively, one may ``steer'' the flow by allowing the surfaces to smoothly move between various $t= $ constant slices, requiring $t$ to restrict to a constant on each surface.  Generally, such a flow will have $\beta$ (defined in \ref{eqn_gen_imcvf}) variable on each surface.

Second,  spherically symmetric spheres in a spherically symmetric spacetime (which are time flat), may be evolved 
with uniformly area expanding velocity and $t$ restricting to a constant on each sphere.  In this way, any spherically symmetric hypersurface of the spacetime can be foliated by such a flow (provided the spheres have inward-spacelike mean curvature vector). 
These examples suggest the necessity of some gauge condition
to determine a unique flow.  One possible condition would be to require
\begin{equation*}
\int_\Sigma \langle T, \nu \rangle dA = 0,
\end{equation*}
on each surface, where $T$ is some fixed future-timelike unit vector field on $N$.

The two most important issues surrounding this flow are existence and asymptotics at infinity.  
In terms of short-time existence, the analytic difficulties appear to be two-fold.  First, the time flat condition involves fourth derivatives of the embedding function, whereas most well-known flows are second-order.  Second, preserving the time flat condition is non-local in character.  

For long-time existence, it is expected that outward ``jumps'' of the surface within the spacetime must occur.  Indeed, the jumping phenomenon is present in the work of Huisken and Ilmanen for inverse mean curvature flow in the time-symmetric case.  A possible condition is that a surface instantaneously jumps to the outermost time flat surface of equal of less area that encloses it\footnote{We say $\Sigma_2$ encloses $\Sigma_1$ if there exists a smooth spacelike hypersurface $\Omega$ with boundary components $\Sigma_1$ and $\Sigma_2$, such that the outward and inward directions of $\Sigma_1$ and $\Sigma_2$, respectively, point into $\Omega$.}. 
Without the time flat condition, this jumping criterion is meaningless:  any spacelike surface in $\R^{3,1}$ is enclosed by surfaces of arbitrarily small area.  For jumps to make sense, it would be necessary to include surfaces that are time flat in a weak sense (in analogy with Huisken--Ilmanen's jump surfaces that have regions of vanishing mean curvature). An open question is to determine whether the Hawking mass can decrease at such a jump. 
\begin{definition}
An embedded surface $\Sigma$ in a spacetime is $C^{k,\alpha}$ \emph{weakly time flat} if it is the $C^{k,\alpha}$-limit of a sequence of time flat surfaces.
\end{definition}
The choice of $k$ and $\alpha$ may depend on the application.  
Such a surface could have null mean curvature vector or a null tangent plane.  

Finally, we address possible behavior of the asymptotics of a spacelike uniformly area expanding time flat flow.  
We propose that the time flat  condition keeps time-like oscillations in check, while the uniformly area expanding condition smooths/regularizes the surfaces in spacelike directions (which occurs in codimension-1 inverse mean curvature flow).  If these heuristics could be justified, the surfaces would conceivably have good asymptotics at infinity.  For instance, if the spacetime is asymptotically flat, then the surfaces may sweep out an asymptotically flat hypersurface with appropriately decaying second fundamental form (and moreover detect the ADM mass or ADM energy).

This discussion leads to the following question:
\begin{question}
For what spacetimes is it possible to construct a uniformly area expanding time flat flow $\{\Sigma_\lambda\}$, beginning at a specified time flat surface $\Sigma$, with $m_H(\Sigma_\lambda)$ limiting to the ADM mass or ADM energy as $\lambda \to \infty$?
\end{question}
For instance, the answer may include perturbations of the Minkowski and Schwarz\-schild spacetimes.

We close with the following comment: the Hawking mass and Brown--York mass  \cite{brown_york} are two highly celebrated
ways of understanding quasi-local mass in general relativity --- yet they are quite different from each other. The Brown--York mass was generalized to spacetimes by Wang and Yau \cites{wang_yau1, wang_yau2} in recent years.  Remarkably, the time flat condition is related to both the spacetime Hawking mass and the Wang--Yau mass!

\section*{Appendix: Geometric computations}
This appendix contains proofs of a number of identities  used in sections \ref{sec_monotonicity} and \ref{sec_time_flat} to establish the variation formula for the Hawking mass.

Equations (\ref{eqn_var_dA}) and (\ref{eqn_var_A}) follow from the first variation of area formula and the fact that $\{\Sigma_\lambda\}$ solve inverse mean curvature flow in $M$ (see \cite{imcf} for instance);  (\ref{eqn_var_H}) is standard and follows from the second variation of area formula,
combined with the Gauss equation traced twice.  

\begin{proof}[Proof of identity (\ref{eqn_var_p}):]
Consider a flow of surfaces $\Sigma_\lambda = \Phi_\lambda(\Sigma)$ in $M$ with unit normals $\nu_\lambda$ and normal velocity $\vec \xi = \eta \nu_0$ at $\lambda=0$, where $\eta$ is some function on $\Sigma=\Sigma_0$.  Recall that
\begin{equation}
\label{eqn_nu_lambda}
\dot \nu_0 :=\left.\frac{\partial \nu_\lambda }{\partial \lambda}  \right|_{\lambda=0} = -\nabla^\Sigma \eta.
\end{equation}
To prove this, first observe that differentiating $\langle \nu_\lambda, \nu_\lambda\rangle=1$ implies $\dot \nu_0$ is orthogonal to $\nu_0$.  Second, define the family of embeddings $F_\lambda : \Sigma_0 \to N$:
$$F_\lambda(x) = \exp_x(\lambda \eta \nu),$$
where $\exp$ is the exponential map of $M$. 
For the purposes of first derivative calculations at $\lambda=0$, we may use $F_\lambda$ in place of $\Phi_\lambda$.  For any fixed vector $X \in T_x \Sigma_0$, we let $X_\lambda = d(F_\lambda)_x(X)$. A standard calculation shows
$$\langle\dot X_0,\nu_0\rangle :=\left\langle\left.\frac{\partial X_\lambda}{\partial \lambda}\right|_{\lambda = 0}, \nu_0\right\rangle = \langle \nabla^\Sigma \eta, X \rangle.$$
It follows that
$$0 = \left.\frac{\partial }{\partial \lambda}\right|_{\lambda = 0}  \langle X_\lambda, \nu_\lambda\rangle = \langle \dot X_0, \nu_0\rangle + \langle X, \dot \nu_0\rangle.$$
Putting this all together, one obtains (\ref{eqn_nu_lambda}).

Therefore, for the present case of inverse mean curvature flow,
\begin{align*}
\left.\frac{\partial p(\nu_\lambda, \nu_\lambda)}{\partial \lambda}  \right|_{\lambda=0} 
&= (\nabla^M_{\vec \xi} p)(\nu,\nu) + 2p\left(\nu, -\nabla^\Sigma \eta\right)\\
&= (\nabla^M_{\nu/H} p)(\nu,\nu) - 2p\left(\nu, \nabla^\Sigma \left(\frac{1}{H}\right)\right)\\
&= \frac{1}{H}(\nabla^M_{\nu} p)(\nu,\nu) + \frac{2}{H^2}p\left(\nu, \nabla^\Sigma H\right),
\end{align*}
from which  (\ref{eqn_var_p}) follows.
\end{proof}

\begin{proof}[Proof of identity \ref{identity_div_M}:]
Let $\{e_1, e_2\}$ be
a local orthonormal frame on $\Sigma$.  We have
\begin{align*}
(\div_M (p))(\nu) &= (\nabla^M_\nu p)(\nu,\nu) + \sum_{i=1}^2 (\nabla^M_{e_i} p)(e_i,\nu)\\
&=(\nabla^M_\nu p)(\nu,\nu) +\sum_{i=1}^2 e_i (p(e_i,\nu)) - p(\nabla^M_{e_i} e_i,\nu) - p(e_i, \nabla_{e_i}^M \nu).
\end{align*}

On the other hand, we  compute the divergence of the 1-form $\ol p$:
\begin{align*}
 \div_\Sigma (\ol p) &= \sum_{i=1}^2 (\nabla^\Sigma_{e_i} p) (e_i)\\
& = \sum_{i=1}^2 e_i (\ol p(e_i)) - \ol p (\nabla^\Sigma_{e_i} e_i)\\
& = \sum_{i=1}^2 e_i ( p(e_i,\nu)) - p(\nabla^\Sigma_{e_i} e_i,\nu).
\end{align*}
Putting these together:
\begin{align*}
(\div_M (p))(\nu) &=(\nabla^M_\nu p)(\nu,\nu) +\div_\Sigma (\ol p) - \sum_{i=1}^2 p((\nabla^M_{e_i} e_i)^{\text{nor}},\nu) + p(e_i, \nabla_{e_i}^M \nu),
\end{align*}
where $(\cdot)^{\text{nor}}$ denotes orthogonal projection onto the normal bundle of $\Sigma$ in $M$.  Recalling the conventions (\ref{eqn_k})--(\ref{eqn_H}) for the second fundamental form and mean curvature, we have
$$(\nabla^M_{e_i} e_i)^{\text{nor}} = -H \nu,$$
 and
\begin{align*}
\sum_{i=1}^2 p(e_i, \nabla_{e_i}^M \nu) &= \sum_{i,j=1}^2 p(e_i, \langle\nabla_{e_i}^M \nu,e_j \rangle e_j)\\
&=\sum_{i,j=1}^2 p(e_i,e_j) A(e_i,e_j)\\
&=\langle p_\Sigma, A\rangle_\Sigma.
\end{align*}
Substituting yields (\ref{identity_div_M}).
\end{proof}

\begin{proof}[Proof of identity (\ref{identity_R}):]
From the definition of $p$ (equation (\ref{eqn_p})), we have
\begin{align*}
\tr_M\! k &=\frac{1}{2} \tr_M\! p\\
k &= \frac{1}{2} (\tr_M \!p)g - p\\
|k|^2 &= |p|^2 - \frac{1}{4}(\tr_M\! p)^2.
\end{align*}
These expressions, substituted into the constraint equation (\ref{constraint1}) gives:
\begin{align*}
R&= 16\pi \mu  +|p|^2- \frac{1}{2}(\tr_M \!p)^2\\
 &= 16\pi \mu  + \left(|p_\Sigma|_\Sigma^2 + 2|\ol p|^2 + p(\nu,\nu)^2\right) - \frac{1}{2} \left( \tr_\Sigma p_\Sigma + p(\nu,\nu)\right)^2\\
 &= 16\pi \mu  +|p_\Sigma|_\Sigma^2 - \frac{1}{2} ( \tr_\Sigma p_\Sigma)^2+ 2|\ol p|^2 +\frac{1}{2} p(\nu,\nu)^2 - p(\nu,\nu)\tr_\Sigma p_\Sigma.
\end{align*}
\end{proof}

Next, (\ref{identity_J}) follows  from the definition of $p$ and the second constraint equation, (\ref{constraint2}).

\begin{proof}[Proof of \ref{identity_p_omega}:]
Observe $\nu^\perp$ is a unit normal to $M$ inside $N$, and recall the sign convention for $k$
from (\ref{eqn_k}).  Then for any vector $X$ tangent to $\Sigma$,
\begin{align*}
\ol p(X) &= p(X,\nu)\\ &= - k(X,\nu)\\
&= \langle \nabla^N_X \nu, \nu^\perp\rangle\\
&= \langle \nabla^\perp_X \nu, \nu^\perp\rangle\\
&=\alpha(X).
\end{align*}

\end{proof}

\end{document}